\documentclass[10pt]{article}
\usepackage{geometry}                
\usepackage{graphicx}
\usepackage{fullpage}
\usepackage{amsmath}
\usepackage{amssymb}
\usepackage{amsthm}
\usepackage{url}
\usepackage{epsfig}
\usepackage{color}
\usepackage{refcount}
\usepackage{verbatim}
\usepackage{enumerate}
\usepackage{multirow}
\usepackage{framed}
\usepackage{slashbox}
\usepackage{caption}
\usepackage{subcaption}
\usepackage{algorithmic}
\usepackage[ruled,vlined]{algorithm2e}
\usepackage[normalem]{ulem}

\usepackage[square,numbers,sort&compress]{natbib}

\usepackage{xr}
\usepackage[breaklinks=true,colorlinks,citecolor=blue,linkcolor=blue]{hyperref}

\usepackage{calc}

\usepackage{tikz}
\usepackage{xifthen}
\usetikzlibrary{calc}
\usepackage{epstopdf}
\usetikzlibrary{positioning}
\usepackage{pgffor}

\usepackage{color}
\definecolor{clemson-orange}{RGB}{234,106,32}
\definecolor{chicago-maroon}{RGB}{128,0,0}
\definecolor{cincinnati-red}{RGB}{190,0,0}
\definecolor{soft-cyan}{RGB}{68,85,90}
\usepackage{fullpage}
\usepackage{multicol}

\usepackage[utf8]{inputenc}

\usepackage{array}
\newcolumntype{L}[1]{>{\raggedright\let\newline\\\arraybackslash\hspace{0pt}}m{#1}}
\newcolumntype{C}[1]{>{\centering\let\newline\\\arraybackslash\hspace{0pt}}m{#1}}
\newcolumntype{R}[1]{>{\raggedleft\let\newline\\\arraybackslash\hspace{0pt}}m{#1}}

\newcommand{\bb}{\mathbb}
\newcommand{\R}{\bb R}
\newcommand{\Z}{{\bb Z}}
\newcommand{\N}{{\bb N}}

\theoremstyle{definition}
\newtheorem{theorem}{Theorem}[section]
\newtheorem{lemma}[theorem]{Lemma}
\newtheorem{corollary}[theorem]{Corollary}
\newtheorem{prop}[theorem]{Proposition}

\newtheorem{definition}[theorem]{Definition}

\newtheorem{example}[theorem]{Example}

\newtheorem{conjecture}[theorem]{Conjecture}

\usepackage{lineno}
\makeatletter
\newcommand*\patchAmsMathEnvironmentForLineno[1]{%
  \expandafter\let\csname old#1\expandafter\endcsname\csname #1\endcsname
  \expandafter\let\csname oldend#1\expandafter\endcsname\csname end#1\endcsname
  \renewenvironment{#1}%
     {\linenomath\csname old#1\endcsname}%
     {\csname oldend#1\endcsname\endlinenomath}}%
\newcommand*\patchBothAmsMathEnvironmentsForLineno[1]{%
  \patchAmsMathEnvironmentForLineno{#1}%
  \patchAmsMathEnvironmentForLineno{#1*}}%
\AtBeginDocument{%
\patchBothAmsMathEnvironmentsForLineno{equation}%
\patchBothAmsMathEnvironmentsForLineno{align}%
\patchBothAmsMathEnvironmentsForLineno{flalign}%
\patchBothAmsMathEnvironmentsForLineno{alignat}%
\patchBothAmsMathEnvironmentsForLineno{gather}%
\patchBothAmsMathEnvironmentsForLineno{multline}%
}

\DeclareMathOperator*{\conv}{conv}

\DeclareMathOperator*{\argmax}{argmax}

\DeclareMathOperator*{\poly}{poly}

\def\ve#1{\mathchoice{\mbox{\boldmath$\displaystyle\bf#1$}}
{\mbox{\boldmath$\textstyle\bf#1$}}
{\mbox{\boldmath$\scriptstyle\bf#1$}}
{\mbox{\boldmath$\scriptscriptstyle\bf#1$}}}

\newcommand{\x}{{\ve x}}

\renewcommand{\a}{{\ve a}}

\numberwithin{equation}{section}  

\title{Enumerating integer points in polytopes with bounded subdeterminants}
\author{Hongyi Jiang\thanks{Department of Applied Mathematics and Statistics, Johns Hopkins University, Baltimore, MD, USA ({\tt hjiang32@jhu.edu}, {\tt basu.amitabh@jhu.edu}). Supported by the ONR Grant N000141812096, NSF Grant CCF2006587 and the AFOSR Grant FA95502010341.} 
\and Amitabh Basu\footnotemark[1]}
\date{\today}
\begin{document}

\maketitle

\begin{abstract} We show that one can enumerate the vertices of the convex hull of integer points in polytopes whose constraint matrices have bounded and nonzero subdeterminants, in time polynomial in the dimension and encoding size of the polytope. This extends a previous result by Artmann et al. who showed that integer linear optimization in such polytopes can be done in polynomial time.\end{abstract}

\section{Integer points in polytopes}
Understanding the structure of integer points in polytopes is a central question in discrete mathematics and geometry of numbers. In the last 50 years, several algorithmic breakthroughs have been made in three fundamental questions, listed in increasing order of difficulty: 1) testing if a given polytope contains an integer point, 2) finding the optimum integer point in a polytope with respect to a linear objective function, 3) enumerating or counting integer points in a polytope. It is well-known that even question 1) is NP-complete, without making further assumptions. Nevertheless, significant progress has been made in understanding under what conditions polynomial time algorithms can be designed for these three tasks. Two prominent research directions have involved investigating the fixed dimension case and the bounded subdeterminant case. To make things precise, consider polytopes given by $\{x\in \R^n: Ax \leq b\}$ where $A\in \Z^{m\times n}$ and $b \in \Z^m$. We have assumed integer data here as will not be considering nonrational polytopes (nevertheless, there is some subtlety involved with integer data in the bounded subdeterminant analysis that we will point out below {-- see the discussion after Lemma~\ref{lem::width-prop}}). 
Two parameters that have received a lot of attention are the dimension $n$ and the maximum absolute value of any $n\times n$ subdeterminant of $A$, denoted in this paper by $\Delta_A$ (some authors have also worked with the maximum $k\times k$ subdeterminant of $A$ over all possible $k\in \{1, \ldots, n\}$). Other parameters have also been studied extensively~\cite{eisenbrand2019algorithmic,onn2010nonlinear,de2012algebraic} (this is a very small sample biased towards monographs and very recent work). We will focus on $n$ and $\Delta_A$ in this paper.

\subsection{Fixed dimension}

Lenstra~\cite{Lenstra83} sparked an active line of research by showing that the linear optimization problem over integer points in a polytope can be solved in polynomial time, if we focus on the family of polytopes in some fixed dimension $n$. The original running time obtained by Lenstra was $2^{O(n^3)}\cdot\poly(n,size(A,b,c))$, where $size(A,b,c)$ denotes the total {binary} encoding size of $A,b$ and the objective vector $c\in \R^n$. {For the definition of binary encoding sizes, see for example \cite[Section 2.1]{SchrijverBOOK}}. Subsequent refinements and improvements have steadily appeared~\cite{Kannan87,GroetschelLovaszSchrijver-Book88,Dadush12,hildebrand2013new,KhachiyanPorkolab00,Heinz05,dadush2011enumerative,cook-hartmann-kannan-mcdiarmid-1992}. The best dependence of the running time on $n$ is currently $2^{O(n\log n)}$ and one of the outstanding open questions in the area is to decide if this can be improved to $2^{O(n)}$. 

The counting problem was also shown to be polynomial time solvable in fixed dimension, starting with the seminal work of Barvinok~\cite{Barvinok94}. Several improvements and variations on Barvinok's insights have been obtained since then; see~\cite{de2012algebraic} for a survey.

\subsection{Bounded subdeterminants}

A classical result in polyhedral combinatorics states that if $\Delta_A = 1$, then all vertices of the polytope are integral~\cite{conforti2014integer}. Thus, using linear programming algorithms, one can solve the integer optimization problem $\max\{c\cdot x: Ax \leq b\}$ in time polynomial in $n$ and the encoding size of $A,b$ and $c$. Veselov and Chirkov had the remarkable insight that if $\Delta_A =2$, then the feasibility problem can be solved in polynomial time~\cite{veselov2009integer}. Artmann, Weismantel and Zenklusen used deep results from combinatorial optimization to establish that the linear optimization problem can also be solved in strongly polynomial time if $\Delta_A=2$~\cite{artmann2017strongly}. The polynomial time solvability of the feasibility, optimization or counting questions for the family of polytopes with $\Delta_A$ bounded by a constant has been another long standing open question in discrete optimization. See~\cite{artmann2016note,gribanov2016width,gribanov2016integer} for some more recent work in this direction.

\section{Our contribution}\label{sec:our-results}

Our main result is the following.\footnote{After this paper was posted on \url{arxiv.org}, Dr. Joseph Paat informed us through personal communication of an alternate proof of this result that uses some recent results on mixed-integer reformulations of integer programs~\cite{paat2020integrality,paat2021integrality}. We give an outline of Dr. Paat's arguments in Section~\ref{sec:conclude}.}.

\begin{theorem}\label{thm:main} Let $\Delta \in \N$ be a fixed natural number. Consider the family of polytopes $P:= \{x\in \R^n : Ax \leq b\}$ where $A\in \Z^{m\times n}$ is such that $\Delta_A \leq \Delta$ and all $n\times n$ subdeterminants of $A$ are nonzero. One can enumerate all the vertices of the convex hull of integer points in $P$ in time polynomial in $n$ and encoding size of $A$ and $b$.
\end{theorem}

This result does not fully resolve the open question of solving integer optimization with bounded subdeterminants in polynomial time because of the nontrivial restriction that all $n\times n$ minors of $A$ have to be nonzero. Nevertheless, Theorem~\ref{thm:main} improves upon the result in~\cite{artmann2016note} where the authors give a polynomial time algorithm for the integer optimization problem under the same hypothesis. We strengthen that result by showing that one can actually enumerate all the vertices of the integer hull, i.e., the convex hull of integer points in the polytope, in polynomial time. This is in line with results cited above from~\cite{cook-hartmann-kannan-mcdiarmid-1992} and~\cite{Barvinok94}, but in the bounded subdeterminant regime instead of the fixed dimension setting.

Theorem~\ref{thm:main} follows from the following three results established in this paper that we believe are of independent interest. We begin by recalling the notion of width.

\begin{definition}\label{def:width}
Given a set $S\subseteq \R^n$ and a vector $v \in \R^n$, the {\em width} of $S$ in the direction $v$ is defined as $$w(v, S) := \max_{x \in S} v \cdot x - \min_{x\in S} v \cdot x.$$ If $P := \{x \in \R^n:Ax\leq b\}$, we use the notation $w(v, A, b)$ to denote the width of $P$ in the direction $v$. If $v$ is a row of $A$ defining a facet, then $w(v, A, b)$ will be called the corresponding {\em facet width}. 
\end{definition}

\begin{theorem}\label{thm::bounded-width-all} Let $S := \{x \in \R^n: Ax \leq b\}$ be a full-dimensional simplex, with $A \in \Z^{(n+1)\times n}$ and $b \in \R^{n+1}$ with the facet width of the first $n$ facets bounded by $W$. Then the number of integer points in $S$ is polynomial in $n$, if $W$ is a fixed constant independent of $n$. Moreover, there is a polynomial time algorithm that enumerates all the integer points.
\end{theorem}


{\begin{theorem}\label{thm::small-simplex}
There exists a function $f: \N \to \N$ with the following property. For any full-dimensional simplex described by $Ax\leq b$, where $A\in \Z^{(n+1)\times n}$, $b\in \R^{n+1}$ with smallest facet width $W_{min}$, all its facet widths are bounded by $W_{min}f(\Delta_A)$. \end{theorem}}

\begin{theorem}\label{thm::all-vert-final} Let $\Delta \in \N$ be a fixed natural number. Consider the family of simplices $S:= \{x\in \R^n : Ax \leq b\}$ where $A\in \Z^{m\times n}$ and $b\in \Z^m$ such that $1 \leq \Delta_A \leq \Delta$ and smallest facet width greater than or equal to $\Delta -1$. There is an algorithm that enumerates all the vertices of the integer hull of $S$ in time polynomial in $n$ and encoding size of $A$ and $b$.
\end{theorem}

We now present a short proof of our main result.

\begin{proof}[Proof of Theorem~\ref{thm:main}] We appeal to the following result from~\cite{artmann2016note}: there exists a constant $C(\Delta)$ such that if $n > C(\Delta)$ then $A$ has at most $n+1$ rows~\cite[Lemma 7]{artmann2016note}\footnote{This is the place where we need the assumption that all $n\times n$ minors are nonzero.}.

If $n \leq C(\Delta)$ then we use the result from~\cite{cook-hartmann-kannan-mcdiarmid-1992} to enumerate the vertices of the integer hull. Else, we know that $P$ is a simplex by the result cited above. {If $P$ is a single point or the empty set, then the result is easy. Else, $P$ must be full-dimensional.} If the smallest facet width is bounded by $\Delta -2$, then Theorems~\ref{thm::small-simplex} and~\ref{thm::bounded-width-all} imply that {\em all} the integer points in $P$ can be enumerated in polynomial time. If the smallest facet width is greater than or equal to $\Delta -1$, then we use the algorithm from Theorem~\ref{thm::all-vert-final}.\end{proof}

{Let us put our results in some context. Several families of polytopes are known in the literature where the number of vertices of the integer hull grows exponentially in the dimension even with bounded subdeterminants, e.g., bipartite matching polytopes. As one sees in the proof of Theorem~\ref{thm:main}, the assumption of nonzero determinants rules out most of these classical examples because it helps to reduce to the case of the simplex. Nevertheless, in \cite[Theorem 1]{barany1992integer}, B\'ar\'any et al. construct a family of simplices with exponential lower bounds on the number of vertices of their integer hulls. Our work shows that with bounded facet width or bounded subdeterminant assumptions (see Theorems~\ref{thm::bounded-width-all},~\ref{thm::small-simplex} and~\ref{thm::all-vert-final} above), one can get around these examples.}

Section~\ref{sec:proofs} presents the proofs of Theorems~\ref{thm::bounded-width-all},~\ref{thm::small-simplex} and~\ref{thm::all-vert-final}. We end in Section~\ref{sec:conclude} with some future directions.

\section{Proofs of Theorems~\ref{thm::bounded-width-all},~\ref{thm::small-simplex} and~\ref{thm::all-vert-final}}\label{sec:proofs}

%
We collect a few simple but useful facts about width.

\begin{lemma}\label{lem::width-prop} The following are all true.
\begin{enumerate}[(a)]
\item $w(v, S) \leq w(v,S')$ for any $v \in \R^n$ and any two sets $S, S' \subseteq \R^n$ such that $S \subseteq S'$.
\item $w(v, S) = w(v, S + t)$ for any $v, t \in \R^n$ and any $S\subseteq \R^n$.
\item $w(v, \alpha S) = |\alpha| w(v, S)$ for any $v \in \R^n$, any $S\subseteq \R^n$ and any $\alpha \in \R$.
\item Let $A\in \Z^{m\times n}$ and $b\in \R^m$. Let $U \in \Z^{n\times n}$ be a unimodular matrix and define $H = AU$. Consider the two polyhedra $\{x: \in \R^n: Ax \leq b\}$ and $\{y \in \R^n: Hy \leq b\}$ which are related by the unimodular transformation given by $x = Uy$. Then the width $w(a, A,b)$ with respect to any row $a$ of $A$ is equal to the width $w(h, H, b)$ given by the corresponding row $h = U^Ta$ of $H$.
\end{enumerate}
\end{lemma}


Let us discuss the integer data assumption here. Typically this is justified by saying that the data is rational and we can scale inequalities to make all the data integer. However, if we wish to impose bounds on the subdeterminants as in this paper, one has to be careful with such scalings. Since the bound is on the subdeterminants of $A$, it is justified to assume that the entries of $A$ are integer. Otherwise, any non-zero bound will be satisfied by every polytope simply by scaling the constraints. However, one may question why $b$ is also assumed to be integer valued. See the final paragraph of Section~\ref{sec:conclude} for some discussion of how this can make a concrete difference. Below, we are careful to impose the integrality assumption on $b$ in our hypotheses only when needed.


\subsection{Proof of Theorem~\ref{thm::bounded-width-all}}

\begin{prop}\label{prop:width-red} Let $Q \subseteq \R^n$ be a simplex {with coefficient matrix $A\in \Z^{m\times n}$} (not necessarily full-dimensional) and let $a\cdot x \leq b$ be a facet defining inequality for $Q$, defining the facet $S$ (also a simplex). Suppose $W_a:= w(a, Q) > 0$. Consider the slice $S' := Q \cap \{x \in \R^n: a\cdot x = b'\}$ for some $b' \in [b-W_a, b]$. Then $S'$ is a translate of $\frac{W_a - (b-b')}{W_a}S$ and consequently $w(v, S') =  \frac{W_a - (b-b')}{W_a}\cdot w(v,S)$ for any $v \in \R^n$. {Furthermore, if $a''\cdot x \leq b''$ defines a facet $S''$ of $Q$ distinct from $S$, then $w(a'', S) = w(a'', Q)$.}
\end{prop}

\begin{proof} Let $v$ be the vertex of $Q$ that does not lie on $S$ (since $W_a>0$). Let the other vertices of $Q$ be given by $v + r_1, \ldots, v + r_k$, for linearly independent vectors $r_1, \ldots, r_k$. Thus, $a\cdot v = b - W_a$ and $a\cdot r_i = W_a$. Using these relations, one can check that $\{v + \frac{W_a - (b-b')}{W_a}r_i: i=1, \ldots, k\}$ satisfy the equation $a\cdot x = b'$. Thus, they are the vertices of the slice $S'$ and {the first part is established. Let $a''\cdot x \leq b''$ define the facet $S''$ distinct from $S$. Let $v+r_1$ be the vertex of $Q$ not on $S''$. Then we have $w(a'', Q)=a''\cdot (v+r_2)-a''\cdot (v+r_1)$. Since $v+r_1$ and $v+r_2$ are vertices of $Q$ on $S$, this implies $w(a'',S)\geq w(a'',Q)$. We already know that $w(r,S)\leq w(r,Q)$ by Lemma~\ref{lem::width-prop} (a). Thus we are done. } 
\end{proof}


\begin{proof}[Proof of Theorem~\ref{thm::bounded-width-all}] We consider a simple enumeration scheme that considers all slices of $S$ parallel to each of the first $n$ facets. More precisely, for each $i=1, \ldots, n$ and $w \in \{0,\ldots, \lfloor W \rfloor\}$,  consider the slice $S^i_w := \{x \in S: a_i\cdot x = \lfloor b_i \rfloor- w\}$, where $a_i \cdot x \leq b_i$ is the $i$-th facet-defining inequality. Since $A \in \Z^{(n+1)\times n}$, all integer points in $S$ are obtained by considering the sets  $$\bigcap_{i=1}^n S^i_{w_i},$$ where we enumerate through the $O(W)$ choices for $w_i \in \{0,\ldots, \lfloor W \rfloor\}$. This gives $O(W^n)$ sets, which is exponential in $n$. However, we will show that most of these sets are actually empty sets, except for a polynomial sized collection. 

To show this, we implement a simple breadth-first search in the style of standard branch-and-bound algorithms {-- see Algorithm~\ref{alg::bounded-width}}. 




\begin{algorithm}
\caption{Enumerating integer points in a simplex}\label{alg::bounded-width}
\begin{algorithmic}
\STATE{Let the root node be $S$ {at depth $0$}.}
\FOR{$i=0:(n-1)$}
\STATE{Step 1: for each nonempty node $N$ at depth $i$, compute the width $W' = w(a_{i+1}, N)$ of $N$ in the direction of the facet normal $a_{i+1}$ (which also defines a facet for $N$).}
\STATE{Step 2: for all $w \in \{0, \ldots, \lfloor W' \rfloor\}$,  Define the children nodes of $N$ at depth $i+1$ as the sets $N \cap \{x: a_{i+1}\cdot x = \lfloor b_{i+1}\rfloor - w\}$.}
\ENDFOR
{\STATE{Report all the nonempty {nodes without children, i.e. nonempty leaves,} in the tree constructed above, that are single integer points.}}
\end{algorithmic}
\end{algorithm}

For any nonempty node $N$ in the tree at depth $i$ that is not a leaf, let $N \cap \{x: a_{i+1}\cdot x = \lfloor b_{i+1}\rfloor\}$ be called the principal child, i.e., $w=0$  in Step 2. of the ``for loop" in Algorithm~\ref{alg::bounded-width}. Note that any node $N$ has at most $W + 1$ children since the width $W' \leq W$ in step 1. of the ``for loop" in Algorithm~\ref{alg::bounded-width} by Lemma~\ref{lem::width-prop} (a) (since $N \subseteq S$ and the facet width of the first $n$ facets of $S$ is bounded by $W$).

{Let $M$ be a node at depth $i$ that is not principal, and $M'$ be its parent node. Then $w(a_{i+1}, M)\leq \frac{w(a_{i+1},M')-1}{w(a_{i+1},M')}\cdot w(a_{i+1},M')\leq \frac{W-1}{W}\cdot w(a_{i+1},M')$ by Proposition~\ref{prop:width-red} and the fact that $w(a_{i+1},M')\leq W$.}


Now let $N$ be a node at depth $i$ in the tree created by Algorithm~\ref{alg::bounded-width}. If the path from the root node to $N$ has $k$ nodes that are not principal, then the facet width $w(a_{i+1}, N)$ of $N$ in the direction of the facet normal $a_{i+1}$  is at most $W\cdot\left(\frac{W-1}{W}\right)^k$, by {the previous paragraph}, Proposition~\ref{prop:width-red}, and Lemma~\ref{lem::width-prop} (b) and (c). If $K = \lceil \frac{\log_2(W+1)}{\log_2(W)-\log_2(W-1)}\rceil$, then $W\cdot\left(\frac{W-1}{W}\right)^K<1$. This implies that for any nonempty node $N$ at depth $i \geq K$, there are at most $K$ nodes on the path from the root to $N$ that are not principal. Using this, we can bound the number of distinct paths from the root to nonempty nodes at depth $i$. Let us first partition the paths {into different classes} according to the levels at which we see nodes that are not principal. There are at most {$\sum_{j=0}^K{i\choose j}$} classes. Within each class, the only variation comes from the branchings at the nodes that are not principal and we have at most $W+1$ children at any internal node. Thus, we have at most {$\sum_{j=0}^K{i\choose j}\cdot (W+1)^j$} distinct such paths, and therefore, nonempty nodes at level $i$. Summing over $i=1, \ldots, n$ levels, we get at most {$n\sum_{j=0}^K{n\choose j}\cdot (W+1)^j$} nonempty nodes in the tree. Since we only branch at nonempty nodes, when we include the infeasible nodes we can increase the count by a factor of at most $W+1$. Thus the overall bound on the number of nodes enumerated by the tree is {$n\sum_{j=0}^K{n\choose j}\cdot (W+1)^{j+1}$}. Since $W$ is fixed, {so is $K$} and this is a polynomial bound (in $n$) as desired. Since linear programming can be used to test emptiness of any node in the tree, the overall algorithm is also polynomial time.
\end{proof}

{We remark here that the idea of using hyperplanes parallel to the facets for enumeration is reminiscent of the proof technique in Cook et al.~\cite{cook-hartmann-kannan-mcdiarmid-1992}. However, there are two important differences. Cook et al. use hyperplanes parallel to the facets for creating polyhedral regions that they search for vertices of the integer hull; we actually use these hyperplanes as ``dual lattice vectors" and consider intersections of these hyperplanes to define single integer points for enumeration. Secondly, our enumeration above gives {\em all} the integer points in the simplex; Cook et al.'s technique produces only the vertices.}

\subsection{Proof of Theorem~\ref{thm::small-simplex}}
\begin{lemma}\label{lem::bounded-elem-inverse}
Let $H\in \Z^{n\times n}$ be a matrix in Hermite Normal Form with $1<\det(H) \leq \Delta$. Further assume $H$ is in the form:
\begin{align}
\begin{bmatrix}
h_{11} & 0 & 0 & \dots & 0\\
h_{21} & h_{22} & 0 & \dots & 0\\
\vdots & \ddots &\ddots & \dots & 0\\
h_{n1} & h_{n2} & \dots &\dots & h_{nn}
\end{bmatrix}.
\end{align}
Let $1 \leq q \leq n$ be such that $h_{ii}=1$ for $i\leq n-q$ and $h_{ii}>1$ for $i\geq n-q+1$, i.e., $q$ of the diagonal entries are strictly bigger than 1. Also, let $h_{ij}'$ be the entry on $i$th row and $j$th column of $H^{-1}$. Then we have:
\begin{enumerate}[(a)]
\item $q \leq \log_2(\Delta)$.
\item Any column of $H^{-1}$ has at most $q+1$ non zero entries.
    \item {$h'_{ij}$ is an integer multiple of $\frac{1}{\det(H)}$.}
    \item $|h'_{ij}|\leq \Delta^{\log_2(\Delta)-1}(\lceil\log_2(\Delta)\rceil)!$.
    \end{enumerate}
 
\end{lemma}
\begin{proof}

Property (a) follows from the fact that the product of the diagonal entries is at most $\Delta$ and thus $2^q \leq \prod_i h_{ii} \leq \Delta$.

Since $h_{ii} = 1$ for $i \leq n- q$ and $H$ is in Hermite Normal Form, $h_{ij}=0$ for $j<i\leq n-q$. Thus, we may write $H$ in the following form
\begin{align}
    \begin{bmatrix}
    I_{(n-q)\times(n-q)} & 0\\
    ~\\
    \widetilde{H} & \widehat{H}
    \end{bmatrix},
\end{align}
where $\widehat{H}$ is a $q\times q$ lower triangular matrix, and $I_{(n-q)\times(n-q)}$ is an $(n-q)\times(n-q)$ identity matrix. This implies the principal $(n-q)\times(n-q)$ minor of $H^{-1}$ must also be $I_{(n-q)\times(n-q)}$. Since $H^{-1}$ must also be lower triangular, $h'_{ij}=0$ for $1\leq i < j \leq n$. From these observations, property (b) follows.

Property (c) follows from Cramer's Rule or the Laplace expansion formula for the inverse {and the fact that $H$ is an integer matrix.}

We next consider Property (d) Since $H^{-1}H$ is an identity matrix, $1\geq h_{ii}' = \frac{1}{h_{ii}}\geq \frac{1}{\Delta}$ for all $i$. We already observed above that $h'_{ij}=0$ for $1\leq i < j \leq n$ and for $i, j \in \{1, \ldots, n-q\}$ except when $i=j$. 

Now consider $i,j\geq n-q+1$. We remove the $i$th column and $j$th row of $H$ to get a matrix $H_{ji}$ and write it as 
\begin{align}
    \begin{bmatrix}
    I_{(n-q)\times(n-q)} & 0\\
    ~\\
    \widetilde{H}_j & \widehat{H}_{ji}
    \end{bmatrix}.
\end{align}
By Cramer's rule, $|h'_{ij}| = \frac{|\det(H_{ji})|}{\Delta}$. Also we have $\det(H_{ji})$ is equal to the determinant of the following matrix
\begin{align}
    \begin{bmatrix}
    I_{(n-q)\times(n-q)} & 0\\
    ~\\
    {\bf 0}_{(n-q)\times q} & \widehat{H}_{ji}
    \end{bmatrix}.
\end{align}
By the definition of the Hermite Normal Form, we have $0\leq h_{ij}\leq \Delta$ for $i>n-q$. Since $q\leq \log_2(\Delta)$, $|\det(\widehat{H}_{ji})|\leq \Delta^{\log_2(\Delta)-1}(\lfloor\log_2(\Delta)\rfloor)!$ {by the Laplace expansion formula of the determinant}. Thus $|h'_{ij}| = \frac{|\det(H_{ji})|}{\Delta}\leq  \frac{|\det(\widehat{H}_{ji})|}{\Delta}\leq \Delta^{\log_2(\Delta)-2}(\lfloor\log_2(\Delta)\rfloor)!$ for $i,j\geq n-q+1$. 

Finally, consider $i\geq n-q+1$ and $j\leq n-q$. Since $H^{-1}H$ is an identity matrix, the inner product of the $i$-th row of $H^{-1}$ and the $j$-th column of $H$ must be $0$. In other words, we have $h'_{ij} + \sum_{k = n-q+1}^nh'_{ik}h_{kj}=0$. Thus $|h'_{ij}| =|\sum_{k = n-q+1}^nh'_{ik}h_{kj}|\leq \Delta^{\log_2(\Delta)-1}(\lceil\log_2(\Delta)\rceil)!$.
\end{proof}

We will also need the following Lemma from \cite{artmann2016note}.
\begin{lemma}\label{lem::bounded-elem}
Given a simplex described by $Ax\leq b$, where $A\in \Z^{(n+1)\times n}$ {is in Hermite Normal Form}, $b\in \R^{n+1}$, and $\Delta_A \leq \Delta$, then the {absolute values of the} entries in $A$ are bounded by a function $g(\Delta)$ which only depends
on $\Delta$.
\end{lemma}

\begin{prop}\label{prop::constant-ratio}
Let $A\in\R^{(n+1)\times n}$ and $b\in\R^{n+1}$, such that $\{x:Ax\leq b\}$ is a full dimensional simplex. 
Let $\hat{A}$ be the first $n$ rows of $A$, $a_i$ be the $i$-th column of $A^T$ and $a_i'$ be the $i$-th column of $\hat{A}^{-1}$. Then we have $\frac{w(a_i, A, b)}{w(a_j, A, b)} =\frac{a_{n+1}^Ta_j'}{a_{n+1}^Ta_i'}$ for $i,j\leq n$.
\end{prop}

\begin{proof}

Let $b'$ be the first $n$ elements of $b$. Consider the vertex $\hat{A}^{-1}b'$ of the simplex. By definition of $w(a_i, A, b)$, $\hat{A}^{-1}b'-w(a_i, A, b)\cdot a'_i$ is the vertex of the simplex that does not lie on the facet given by $a_i\cdot x = b_{i}$. This vertex must lie on the facet given by $a_{n+1}\cdot x = b_{n+1}$. Thus, we have $ a_{n+1}^T(\hat{A}^{-1}b'-w(a_i, A, b)\cdot a'_i) = b_{n+1}$, which implies $w(a_i, A, b) = \frac{b_{n+1}-a_{n+1}^T\hat{A}^{-1}b'}{-a_{n+1}^Ta_i'}$. 
Similarly, we have  $w(a_j, A, b) = \frac{b_{n+1}-a_{n+1}^T\hat{A}^{-1}b'}{-a_{n+1}^Ta_j'}$. Thus $\frac{w(a_i, A, b)}{w(a_j, A, b)} =\frac{a_{n+1}^Ta_j'}{a_{n+1}^Ta_i'}$. 
\end{proof}

\begin{theorem}\label{thm::bounded-ratio}
 Given a full-dimensional simplex described by $Ax\leq b$, where $A\in \Z^{(n+1)\times n}$, $b\in \R^{n+1}$, and $\Delta_A \leq \Delta$, then $\frac{w(a_i, A, b)}{w(a_j, A, b)}\leq g(\Delta)\Delta^{\log_2(\Delta)}(\lceil\log_2(\Delta)\rceil+1)!$, where $g$ is the function from Lemma \ref{lem::bounded-elem}.
\end{theorem}

\begin{proof}
By Lemma~\ref{lem::width-prop} (d), it suffices to prove the result for the simplex $\{y : Hy\leq b\}$ where $H$ is the Hermite Normal Form of $A$. 
By permuting rows of $A$ while computing the Hermite Normal Form, we may assume there exists $1 \leq q \leq n$ such that $h_{ii}=1$ for $i\leq n-q$ and $h_{ii}>1$ for $n-q< i\leq n$, and so, we are in the setting of Lemma~\ref{lem::bounded-elem-inverse}. Thus, $q\leq \log_2(\Delta)$ by Lemma~\ref{lem::bounded-elem-inverse} (a). Let $\widehat{H}$ be the first $n$ rows of $H$. Let $h'_i$ be the $i$th column of $\widehat{H}^{-1}$ and $h_i$ be the $i$th column of $H^T$. Then by Proposition \ref{prop::constant-ratio}, we have that $\frac{w(h_i, H, b)}{w(h_j, H, b)}=\frac{h_{n+1}^Th'_j}{h_{n+1}^Th'_i}$. By Lemma \ref{lem::bounded-elem-inverse} (b), we know that $h'_i$ and $h'_j$ only has at most $q+1$ non-zero elements. Also, by Lemma \ref{lem::bounded-elem}, the entries of $h_{n+1}$ is bounded by $g(\Delta)$. {Combined with Lemma \ref{lem::bounded-elem-inverse} (d), this implies $h_{n+1}^Th'_j\leq g(\Delta)\Delta^{\log_2(\Delta)-1}(\lceil\log_2(\Delta)\rceil+1)!$. Since $h_{n+1}^Th'_i >0$ and all entries of $h'_i$ are integer multiples of $\det(\widehat H)$ by Lemma \ref{lem::bounded-elem-inverse} (c), we must have $h_{n+1}^Th'_i\geq \frac{1}{\det(\widehat H)} \geq \frac{1}{\Delta}$.} Therefore,  $\frac{w(h_i, H, b)}{w(h_j, H, b)}\leq g(\Delta)\Delta^{\log_2(\Delta)}(\lceil\log_2(\Delta)\rceil+1)!$
\end{proof}

\begin{proof}[Proof of Theorem~\ref{thm::small-simplex}]Theorem~\ref{thm::bounded-ratio} implies Theorem~\ref{thm::small-simplex}. \end{proof}
\subsection{Proof of Theorem~\ref{thm::all-vert-final}}
{We first give a lemma that links the inner and outer descriptions of a simplicial cone and the integers points in it.}
\begin{lemma}\label{lem::rep-of-int}
Let $C$ be {a translation of a} simplicial cone  defined by $Ax\leq b$ where $A\in \Z^{n\times n}$ and $b\in \Z^n$. Also, let $a'_1$, $a'_2,\ldots, a'_n$ be the columns of $A^{-1}$, and $u:=A^{-1}b$ be the vertex of $C$. Then for any $v\in C\cap \Z^n$, there exists $\mu := (\mu_1, \ldots, \mu_n)$, where $\mu_i\in \Z$ and $\mu_i\geq 0$ for $1\leq i\leq n$, such that $v = u - \sum_{i=1}^n\mu_ia'_i$. 
\end{lemma}

\begin{proof}
Let $C' := \{x: x = u - \sum_{i=1}^n\mu_i'a'_i, \mbox{ where } \mu_i'\geq 0\mbox{ for } 1\leq i \leq n\}$. Consider any $x = u - \sum_{i=1}^n\mu_i'a'_i, \mbox{ where } \mu_i' \in \R \mbox{ for } 1\leq i \leq n\in C'$. We have 
\begin{align}
    Ax = A(u - \sum_{i=1}^n\mu_i'a'_i) = b - 
    \begin{pmatrix}
    \mu_1'\\
    \mu_2'\\
    \dots\\
    \mu_n'
    \end{pmatrix}.
\end{align}
Thus, $x \in C$, i.e., $Ax \leq b$ if and only if $\mu' \geq 0$. Therefore, $C = C'$. For any $v\in C\cap \Z^n$, express $v = u - \sum_{i=1}^n\mu_ia'_i$, where $\mu_i\geq 0$ for $1\leq i \leq n$. Then the same calculation as above yields $Av = b - \mu$. 
$Av\in \Z^n$ since $v\in \Z^n$ and $A\in \Z^{n\times n}$. Since $b\in \Z^n$, this implies that $\mu\in \Z^n$, and the proof is finished.
\end{proof}

\begin{lemma}\label{lem::rep-of-vert}
With the same notation as Lemma \ref{lem::rep-of-int}, let $\det(A) = \Delta$, and $X$ be the convex hull of the set $C\cap \Z^n$. If $v = u - \sum_{i=1}^n\mu_ia'_i$, where $\mu_i\geq 0$ for $1\leq i \leq n$ is a vertex of $X$, then we have $\prod_{i=1}^{n}(\mu_i+1)\leq \Delta$.
\end{lemma}
\begin{proof}
We will prove this by contradiction. Assume there exists a vertex $v$ of $X$ such that $v = u - \sum_{i=1}^n\mu_ia'_i$, where $\mu_i\geq 0$ for $1\leq i \leq n$ and $\prod_{i=1}^{n}(\mu_i+1)> \Delta$. 
Due to the fact that $\det(A^{-1})=\Delta^{-1}$, the columns of $A^{-1}$ define a lattice $L$ such that $\Z^n\subseteq L$, and $|L/\Z^n| = \Delta$, i.e., there are $\Delta$ cosets with respect to the sublattice $\Z^n$ of $L$. Also, $u\in L$ since $u=A^{-1}b$ and $b \in \Z^n$. 
Then since $\prod_{i=1}^{n}(\mu_i+1)> \Delta$, by the pigeon hole principle, there exists $x_1 = u - \sum_{i=1}^n\mu_i'a'_i$ and $x_2 = u - \sum_{i=1}^n\mu''_ia'_i$, such that $0\leq \mu_i'\leq \mu_i$, $0\leq \mu_i''\leq \mu_i$ for $1\leq i\leq n$, $x_1\neq x_2$, and $x_1- x_2\in \Z^n$.
Then, $v+ (x_1 - x_2)$ and $v-(x_1 - x_2)$ are both in $C\cap \Z^n$ and therefore in $X$, contradicting the fact that $v$ is a vertex of $X$. 
\end{proof}

\begin{theorem}\label{thm::all-vert-1}
 With the same notation as in Lemma \ref{lem::rep-of-int} and Lemma \ref{lem::rep-of-vert}, let $S$ be the simplex given by the convex hull of $\{u, u - (\Delta-1)a'_1, u - (\Delta-1)a'_2,\ldots,u - (\Delta-1)a'_n\}$. If $v = u - \sum_{i=1}^n\mu_ia'_i$ is a vertex of $X$, then $v\in S$.
\end{theorem}
\begin{proof}
By Lemma \ref{lem::rep-of-vert}, we have $\prod_{i=1}^{n}(\mu_i+1)\leq \Delta$. Without loss of generality let $1\leq \mu_{1}\leq \mu_{2}\leq \ldots\leq \mu_{K}$, and the others are $0$. 

{\bf Claim} $K\cdot \mu_K\leq \Delta - 1$.
\begin{proof}
We will prove the claim by induction. When $K=1$, this is trivial. Assume it is true for $K=K_0\geq 1$. Consider $K=K_0+1$. Let $\Delta' = \prod_{i=2}^K (\mu_i+1)$. Then we have $\Delta-1 = \Delta'\cdot(\mu_1+1)-1\geq (\mu_K\cdot(K-1)+1)(\mu_1+1) -1\geq 2\mu_K\cdot(K-1) + 2 -1\geq \mu_K\cdot K$, where the first inequality follows from the induction hypothesis, the second inequality follows from the fact that $\mu_1 \geq 1$ and the final inequality follows from the fact that $K \geq 2$. 
\end{proof}

This claim implies that $v = \frac{1}{K}\sum_{i=1}^K(u-\mu_iKa'_i)\in S$, which finishes the proof. 
\end{proof}

%
%

{The conclusions and techniques of Lemma~\ref{lem::rep-of-vert} and Theorem~\ref{thm::all-vert-1} have appeared in the literature before, although in slightly different language; see, e.g.,~\cite{MR0182454,wolseyBook,sch,conforti2014integer}. We include our particular versions and proofs to keep the paper self-contained.}

We now have all the pieces together to design an algorithm that enumerates a polynomial sized superset of all the vertices of the integer hull.


\begin{algorithm}[H]
\caption{Vertices of the integer hull}\label{alg::large-width}
\begin{algorithmic}
\STATE{\textbf{Input:} A simplex $S= \{x\in \R^n: Ax\leq b\}$ with $\Delta_A \leq \Delta$, and smallest facet width greater than or equal to $\Delta -1$.}
\STATE{
\textbf{Output:} A set $V$ of {cardinality} polynomial in $n$ that includes all the vertices of {the} integer hull of $S$.}
\STATE{Let $A_1x\leq b^{(1)}$, $A_2x\leq b^{(2)}\ldots,A_{n+1}x\leq b^{(n+1)}$ be all the combinations of $n$ inequalities of $Ax\leq b$.}
\STATE{Initialize $V$ as an empty set.}
\FOR{$i=1:(n+1)$}
\STATE{Compute $u=A_i^{-1}b^{(i)}$. Let $a_j'$ denote the $j$-th the column of $A_i^{-1}$.}
\STATE{Let $S_i$ be the convex hull of the set $\{u, u - (\Delta-1)a'_1, u - (\Delta-1)a'_2,\ldots,u - (\Delta-1)a'_n\}$}
\STATE{Apply Algorithm \ref{alg::bounded-width} to get all the integer points in $S_i$ and include them in $V$.}
\ENDFOR
\end{algorithmic}
\end{algorithm}

\begin{theorem}\label{thm::all-vert}
 The set $V$ computed in Algorithm~\ref{alg::large-width} includes all vertices of the convex hull of $S\cap \Z^n$.  
\end{theorem}
\begin{proof}
We use the same notation as in Algorithm~\ref{alg::large-width}.  Consider a vertex $v$ of the convex hull of $S\cap \Z^n$. Let $c\in \R^n$ be an objective vector such that $v$ is the unique solution to 
$$
    \underset{x\in S\cap \Z^n}{\argmax} ~c^Tx.
$$
There exists {an} $i$ such that $A_i^{-1}b^{(i)}$ is the solution to
$$
    \underset{x\;:\; A_ix\leq b^{(i)}}{{\argmax}} ~c^Tx.
$$
Since the facet width of $S$ is at least $\Delta - 1 $, we have $S_i \subseteq S \subseteq \{x\in \R^n:A_ix\leq b^{(i)}\}$. Thus, 

\begin{equation}\label{eq:SiPC}\underset{x\in S_i\cap \Z^n}{\max} ~c^Tx \leq \underset{x\in S\cap \Z^n}{\max} ~c^Tx \leq \underset{\substack{A_ix\leq b^{(i)}\\x\in \Z^n}}{\max} ~c^Tx.\end{equation}
On the other hand, by Theorem~\ref{thm::all-vert-1}, $S_i$ contains all the vertices of convex hull of $\{x\in \Z^n:A_ix\leq b^{(i)}\}$. Thus, all three inequalities in~\eqref{eq:SiPC} are actually equalities. Since $v$ is the unique solution to ${\argmax}\{c^Tx: x\in S\cap \Z^n\}$ and $S_i\subseteq S$, we see that $v \in S_i$.
\end{proof}

%

\begin{proof}[Proof of Theorem~\ref{thm::all-vert-final}]
Let $V=\{v_1, v_2, \ldots, v_\alpha\}$ be the set computed by Algorithm \ref{alg::large-width} with $\alpha:= |V|$. The number of integer points in each $S_i$ in Algorithm~\ref{alg::large-width} is polynomial in $n$ by Theorem~\ref{thm::small-simplex} since the facet widths of the first $n$ facets of each $S_i$ are at most $\Delta - 1$. Thus, $\alpha$ is a function of $n$ and $\Delta$, and polynomial in $n$. 
To check whether for a given $i \in \{1, \ldots, \alpha\}$, $v_i$ is a vertex of the convex hull of $V$ (which is the same as the convex hull of $S\cap \Z^n$), we just need to check the feasibility of the following polynomially many constraints on $\mu_1, \ldots, \mu_\alpha$:
\begin{align*}
    &v_i = \sum_{\substack{j=1\\j\neq i}}^\alpha \mu_jv_j, \qquad\sum_{\substack{j=1\\j\neq i}}^\alpha\mu_j=1, \qquad \mu_j\geq 0\mbox{ for }j\neq i
\end{align*} Solving these polynomially many linear programs, one can filter out the vertices of the integer hull of $S$ from $V$.
\end{proof}

\section{Concluding Remarks}\label{sec:conclude}

A similar argument as the proof of Theorem~\ref{thm::all-vert} gives the following result which we believe to be interesting because it shows a connection between the integer hull of a simplex and the corner polyhedra associated with it~\cite{MR0256718}.

\begin{corollary}\label{cor:corner}
Let $S$ be a simplex described by $Ax\leq b$ where $A\in \Z^{(n+1)\times n}$ and $b\in \Z^{n+1}$ such that $\Delta_A \leq \Delta$, and all its facet widths are greater than or equal to $\Delta -1$. Then the integer hull of $S$ is the intersection of all the integer hulls of the simplicial cones derived by $Ax\leq b$. 
\end{corollary} 

\begin{proof}
Let $P$ be the intersection of all the integer hulls of the $n$ simplicial cones derived by selecting $n$ inequalites from the system $Ax\leq b$. Consider a vertex $v$ of $P$. It suffices to prove that $v\in \Z^n$. Let $c\in \R^n$ be an objective vector such that $v$ is the unique solution to 
\begin{equation}\label{eq::argmaxc}
    \underset{x\in P}{\argmax} ~c^Tx,
\end{equation}
With the similar argument as in Theorem \ref{thm::all-vert}, we can prove that 
\begin{equation}\underset{x\in S_i\cap \Z^n}{\max} ~c^Tx = \underset{x\in P}{\max} ~c^Tx = \underset{\substack{A_ix\leq b^{(i)}\\x\in \Z^n}}{\max} ~c^Tx\end{equation}
for some $i$. Since $S_i\cap \Z^n\subseteq P$, and $v$ is the unique solution to (\ref{eq::argmaxc}), so $v\in S_i\cap \Z^n$.
\end{proof}

In general, there exist simplices such that intersection of the corner polyhedra is a strict superset of the integer hull. Corollary~\ref{cor:corner} says that for ``fat" simplices the intersection is indeed the integer hull (this is also easily seen to hold for simplices with at most one integer point).

The idea from Algorithm~\ref{alg::bounded-width} of enumerating along the facet directions leads us to the following conjecture which we believe is an interesting discrete geometry question. The conjecture is an attempt to generalize the following facts. When $\Delta_A = 1$, the polyhedron defined by $P:=\{x\in \R^n: Ax \leq b\}$ has integral vertices if it is nonempty. When $\Delta_A = 2$, it was shown in~\cite{veselov2009integer} that if $P$ is full dimensional, then $P$ must contain an integer point. One can summarize both these statements by saying that $P\cap \Z^n = \emptyset$ implies the facet width of $P$ is at most $\Delta_A - 2$. 

\begin{conjecture}\label{conj:thin} There is a function $g : \N \to \N$ such that for any $A \in \Z^{m\times n}$ and $b \in \Z^m$, if $\{x \in \Z^n : Ax \leq b\} = \emptyset$, then there is a constraint $\langle a_i, x \rangle \leq b_i$ for some $i \in \{1, \ldots, m\}$ such that $$w(a_i,P) \leq g(\Delta_A),$$ where $P = \{x \in \R^n: Ax \leq b\}$.
\end{conjecture}

In other words, if a polytope has no integer point, then one of its facet widths is bounded by an explicit function of the maximum subdeterminant $\Delta_A$. If this conjecture is true, then by enumerating all the ``slices" in the direction of this facet and recursing on dimension (like in Lenstra-style algorithms), one would obtain an algorithm that decides integer feasibility in time $2^{O(nh(\Delta_A))}\poly(n,size(A,b))$ for some explicit function $h$.
Well-known calculations show that if $\{x \in \Z^n : Ax \leq b\} \neq \emptyset$, then there is a vector $x^\star \in \Z^n$ such that $Ax^\star \leq b$ and each coordinate of $x^\star$ has absolute value at most $n(n+1)\Delta_A$. Thus, a brute force enumeration over the box $[-n(n+1)\Delta_A, n(n+1)\Delta_A]^n$ could work and has complexity $2^{O(n\log_2 n\log_2 \Delta_A)}$. But there does not seem to be an obvious way to improve the $O(n\log_2 n)$ factor to $O(n)$ in the exponent. Thus, Conjecture~\ref{conj:thin} seems to be an intermediate step towards resolving the major open question of designing a $2^{O(n)}$ algorithm for integer optimization. Even without this motivation, we find Conjecture~\ref{conj:thin} to be an intriguing geometric question and worthy of study. Its resolution will give us more insight into how the geometry of a polytope is dictated by its algebraic description.

We emphasize that one needs to impose integrality of the right hand side $b$ in the hypothesis of Conjecture~\ref{conj:thin}; otherwise, the conjecture is false as is shown by the following example.

\begin{example}
Let $I_n$ be an $n\times n$ identity matrix, $a = (1,1,1,\ldots, 1)^T\in \R^n$, and $b = (-\frac{1}{2}, -\frac{1}{2}, \ldots, -\frac{1}{2}, n+\frac{1}{2})^T\in \R^{n+1}$. Then let $P$ be described by 
\begin{align}
\begin{bmatrix}
-I_n\\a^T
\end{bmatrix}x\leq b.
\end{align}
$P$ is a full-dimensional simplex with subdeterminants bounded by 1 and its smallest facet width is $O(n)$, but it does not contain any integer points. 
\end{example}

\noindent{\bf An alternate proof of Theorem~\ref{thm:main}.} As mentioned in Section~\ref{sec:our-results}, our main result can be obtained using completely different tools, as discovered by Dr. Joseph Paat~\cite{paat-personal}. We sketch these arguments here. In~\cite{paat2020integrality,paat2021integrality}, the authors show that under the assumptions of Theorem~\ref{thm:main}, the convex hull of integer points in the polyhedron $\{x\in \R^n : Ax \leq b\}$ is exactly the same as the convex hull of a mixed-integer reformulation: i.e., $$\conv\{x \in \Z^n : Ax \leq b\} = \conv\{x \in \R^n: Ax \leq b, Wx \in \Z^k\},$$ where $W \in \Z^{n\times k}$ and $k$ is a constant depending on $\Delta$. Thus, the vertices of the integer hull of $\{x\in \R^n : Ax \leq b\}$ can be enumerated by enumerating the vertices of $\conv\{x \in \R^n: Ax \leq b, Wx \in \Z^k\}$. The vertices of this latter set can be obtained by enumerating the $k$ dimensional faces of the polyhedron $\{x\in \R^n : Ax \leq b\}$, and then enumerating the vertices of the integer hull (in $\R^n$) of these $k$-dimensional faces. In~\cite[Lemma 8]{paat2021integrality}, the authors also show the number of rows of $A$ is upper bounded by $n + \Delta^2$ under the hypotheses of Theorem~\ref{thm:main}. Thus, the number of these $k$ dimensional faces is upper bounded by ${n+\Delta^2\choose{n-k}} = {n+\Delta^2\choose{\Delta^2 + k}}$ which is polynomial in $n$. The vertices of the integer hull of these $k$ dimensional faces can be enumerated in time polynomial in the encoding sizes of $A$ and $b$, using the algorithm in~\cite{cook-hartmann-kannan-mcdiarmid-1992}, since the dimension $k$ is a constant independent of $n$.

In contrast, our proof is based on different ideas and we believe that the main appeal of our approach is in the three results stated in Theorems~\ref{thm::bounded-width-all},~\ref{thm::small-simplex},~\ref{thm::all-vert-final}.

  \bibliographystyle{plain}
  \bibliography{full-bib}
\end{document}